\documentclass[a4paper, 12pt]{amsart}

\newtheorem{thm}{Theorem}[section]

\newtheorem{lem}[thm]{Lemma}
\newtheorem{prop}[thm]{Proposition}

\newtheorem{mainthm}[thm]{Main Theorem}

\theoremstyle{definition}
\newtheorem{defin}[thm]{Definition}
\newtheorem{rem}[thm]{Remark}

\numberwithin{equation}{section}

\newcommand{\kkk}{\mathbf k}

\newcommand{\zC}{\mathbb C}
\newcommand{\zR}{\mathbb R}
\newcommand{\zN}{\mathbb N}
\newcommand{\zK}{\mathbb K}
\newcommand{\fU}{\mathfrak U}

\newcommand{\cL}{\mathcal L}

\begin{document}


\baselineskip=17pt



\title[On the norm of products of polynomials
on ultraproducts of Banach spaces]{On the norm of products of polynomials
on ultraproducts of Banach spaces}

\author[J. T. Rodr\'{i}guez]{Jorge Tom\'as Rodr\'{i}guez}
\address{IMAS-CONICET}
\email{jtrodrig@dm.uba.ar}
\date{}

\begin{abstract}

The purpose of this article is to study the problem of finding sharp lower bounds for the norm of the product of polynomials in the ultraproducts of Banach spaces $(X_i)_\fU$. We show that, under certain hypotheses, there is a strong relation between this problem and the same problem for the spaces $X_i$.
\end{abstract}

\keywords{Polynomials, Banach spaces, norm inequalities, ultraproducts}

\maketitle

\section{Introduction}

In this article we study the factor problem in the context of ultraproducts of Banach spaces. This problem can be stated as follows: for a Banach space $X$ over a field $\zK$ (with $\zK=\zR$ or $\zK=\zC$) and natural numbers $k_1,\cdots, k_n$ find the optimal constant $M$ such that, given any set of continuous scalar polynomials $P_1,\cdots,P_n:X\rightarrow \zK$, of degrees $k_1,\cdots,k_n$; the inequality
\begin{equation}\label{problema}
M \Vert P_1 \cdots P_n\Vert \ge  \, \Vert P_1 \Vert \cdots \Vert P_n \Vert
\end{equation}
holds, where $\Vert P \Vert = \sup_{\Vert x \Vert_X=1} \vert P(x)\vert$. We also study a variant of the problem in which we require the polynomials to be homogeneous.

Recall that a function $P:X\rightarrow \zK$ is a continuous $k-$homogeneous polynomial if there is a continuous $k-$linear function $T:X^k\rightarrow \zK$ for which $P(x)=T(x,\cdots,x)$. A function $Q:X\rightarrow \zK$ is a continuous polynomial of degree $k$ if $Q=\sum_{l=0}^k Q_l$ with $Q_0$ a constant, $Q_l$ ($1\leq l \leq k$) an $l-$homogeneous polynomial and $Q_k \neq 0$ . 

The factor problem has been studied by several authors. In \cite{BST}, C. Ben\'{i}tez, Y. Sarantopoulos and A. Tonge proved that, for continuous polynomials, inequality (\ref{problema}) holds with constant
\[
M=\frac{(k_1+\cdots + k_n)^{(k_1+\cdots +k_n)}}{k_1^{k_1} \cdots k_n^{k_n}}
\]
for any complex Banach space. The authors also showed that this is the best universal constant, since there are polynomials on $\ell_1$ for which equality prevails. 
For complex Hilbert spaces and homogeneous polynomials, D. Pinasco proved in \cite{P} that the optimal constant is
\begin{equation}\nonumber
M=\sqrt{\frac{(k_1+\cdots + k_n)^{(k_1+\cdots +k_n)}}{k_1^{k_1} \cdots k_n^{k_n}}}.
\end{equation}
This is a generalization of the result for linear functions obtained by Arias-de-Reyna in \cite{A}. In \cite{CPR}, also for homogeneous polynomials, D. Carando, D. Pinasco and the author proved that for any complex $L_p(\mu)$ space, with $dim(L_p(\mu))\geq n$ and $1<p<2$, the optimal constant is
\begin{equation}\nonumber
M=\sqrt[p]{\frac{(k_1+\cdots + k_n)^{(k_1+\cdots +k_n)}}{k_1^{k_1} \cdots k_n^{k_n}}}.
\end{equation}

This article is partially motivated by the work of M. Lindstr\"{o}m and R. A. Ryan in \cite{LR}. In that article they studied, among other things, a problem similar to (\ref{problema}): finding the so called polarization constant of a Banach space. They found a relation between the polarization constant of the ultraproduct $(X_i)_\fU$ and the polarization constant of each of the spaces $X_i$. Our objective is to do an analogous analysis for our problem (\ref{problema}). That is, to find a relation between the factor problem for the space $(X_i)_\fU$  and the factor problem for the spaces $X_i$.

In Section 2 we give some basic definitions and results of ultraproducts needed for our discussion. In Section 3 we state and prove the main result of this paper, involving ultraproducts, and a similar result on biduals.

\section{Ultraproducts}

We begin with some definitions, notations and basic results on filters, ultrafilters and ultraproducts. Most of the content presented in this section, as well as an exhaustive exposition on ultraproducts, can be found in Heinrich's article \cite{H}.

A filter $\fU$ on a family $I$ is a collection of non empty subsets of $I$ closed by finite intersections and inclusions. An ultrafilter is maximal filter.

In order to define the ultraproduct of Banach spaces, we are going to need some topological results first. 

\begin{defin} Let $\fU$ be an ultrafilter on $I$ and $X$ a topological space. We say that the limit of $(x_i)_{i\in I} \subseteq X$ respect of $\fU$ is $x$ if for every open neighborhood $U$ of $x$ the set $\{i\in I: x_i \in U\}$ is an element of $\fU$. We denote
$$ \displaystyle\lim_{i,\fU} x_i = x.$$

\end{defin}

The following is Proposition 1.5 from \cite{H}.

\begin{prop}\label{buenadef} Let $\fU$ be an ultrafilter on $I$, $X$ a compact Hausdorff space and $(x_i)_{i\in I} \subseteq X$. Then, the limit of $(x_i)_{i\in I}$ respect of $\fU$ exists and is unique.
\end{prop} 

Later on, we are going to need the next basic Lemma about limits of ultraproducts, whose proof is an easy exercise of basic topology and ultrafilters.

\begin{lem}\label{lemlimit} Let $\fU$ be an ultrafilter on $I$ and $\{x_i\}_{i\in I}$ a family of real numbers. Assume that the limit of $(x_i)_{i\in I} \subseteq \zR$ respect of $\fU$ exists and let $r$ be a real number such that there is a subset $U$ of $\{i: r<x_i\}$ with $U\in \fU$. Then
$$ r \leq \displaystyle\lim_{i,\fU} x_i. $$

\end{lem}

We are now able to define the ultraproduct of Banach spaces. Given an ultrafilter $\fU$ on $I$ and a family of Banach spaces $(X_i)_{i\in I}$, take the Banach space $\ell_\infty(I,X_i)$ of norm bounded families $(x_i)_{i\in I}$ with $x_i \in X_i$ and norm
$$\Vert (x_i)_{i\in I} \Vert = \sup_{i\in I} \Vert x_i \Vert.$$ The ultraproduct $(X_i)_\fU$ is defined as the quotient space $\ell_\infty(I,X_i)/ \sim $ where
$$ (x_i)_{i\in I}\sim (y_i)_{i\in I} \Leftrightarrow \displaystyle\lim_{i,\fU} \Vert x_i - y_i \Vert = 0.$$

Observe that Proposition \ref{buenadef} assures us that this limit exists for every pair $(x_i)_{i\in I}, (y_i)_{i\in I}\in \ell_\infty(I,X_i)$. We denote the class of $(x_i)_{i\in I}$ in $(X_i)_\fU$ by $(x_i)_\fU$.

The following result is the polynomial version of Definition 2.2 from \cite{H} (see also Proposition 2.3 from \cite{LR}). The reasoning behind is almost the same.

\begin{prop}\label{pollim} Given two ultraproducts $(X_i)_\fU$, $(Y_i)_\fU$ and a family of continuous homogeneous polynomials $\{P_i\}_{i\in I}$ of degree $k$ with 
$$ \displaystyle\sup_{i\in I} \Vert P_i \Vert < \infty,$$
the map $P:(X_i)_\fU \longrightarrow (Y_i)_\fU$ defined by $P((x_i)_\fU)=(P_i(x_i))_\fU$ is a continuous homogeneous polynomial of degree $k$. Moreover $\Vert P \Vert = \displaystyle\lim_{i,\fU} \Vert P_i \Vert$.

If $\zK=\zC$, the hypothesis of homogeneity can be omitted, but in this case the degree of $P$ can be lower than $k$.
\end{prop}

\begin{proof} Let us start with the homogeneous case. Write $P_i(x)=T_i(x,\cdots,x)$ with $T_i$ a $k-$linear continuous function. Define $T:(X_i)_\fU^k \longrightarrow (Y_i)_\fU$ by 
$$T((x^1_i)_\fU,\cdots,(x^k_i)_\fU)=(T_i(x^1_i,\cdots ,x^k_i))_\fU.$$
$T$ is well defined since, by the polarization formula, $ \displaystyle\sup_{i\in I} \Vert T_i \Vert \leq   \displaystyle\sup_{i\in I} \frac{k^k}{k!}\Vert P_i \Vert< \infty$.

Seeing that for each coordinate the maps $T_i$ are linear, the map $T$ is linear in each coordinate, and thus it is a $k-$linear function. Given that 
$$P((x_i)_\fU)=(P_i(x_i))_\fU=(T_i(x_i,\cdots,x_i))_\fU=T((x_i)_\fU,\cdots,(x_i)_\fU)$$
we conclude that $P$ is a $k-$homogeneous polynomial.

To see the equality of the norms for every $i$ choose a norm $1$ element $x_i\in X_i$ where $P_i$ almost attains its norm, and from there  is easy to deduce that $\Vert P \Vert \geq \displaystyle\lim_{i,\fU} \Vert P_i \Vert$. For the other inequality we use that $$ |P((x_i)_\fU)|= \displaystyle\lim_{i,\fU}|P_i(x_i)| \leq \displaystyle\lim_{i,\fU}\Vert P_i \Vert \Vert x_i \Vert^k = \left(\displaystyle\lim_{i,\fU}\Vert P_i \Vert \right)\Vert (x_i)_\fU \Vert^k .$$

Now we treat the non homogeneous case. For each $i\in I$ we write $P_i=\sum_{l=0}^kP_{i,l}$, with $P_{i,0}$ a constant and $P_{i,l}$ ($1\leq l \leq k$) an $l-$homogeneous polynomial. Take the direct sum $X_i \oplus_\infty \zC$ of $X_i$ and $\zC$, endowed with the norm $\Vert (x,\lambda) \Vert =\max \{ \Vert x \Vert, | \lambda| \}$. Consider the polynomial $\tilde{P_i}:X_i \oplus_\infty \zC\rightarrow Y_i$ defined by $\tilde{P}_i(x,\lambda)=\sum_{l=0}^k P_{i,l}(x)\lambda^{k-l}$. The polynomial $\tilde{P}_i$ is an homogeneous polynomial of degree $k$ and, using the maximum modulus principle, it is easy to see that $\Vert P_i \Vert = \Vert \tilde{P_i} \Vert $. Then, by the homogeneous case, we have that the polynomial $\tilde{P}:(X_i \oplus_\infty \zC)_\fU \rightarrow (Y_i)_\fU$ defined as $\tilde{P}((x_i,\lambda_i)_\fU)=(\tilde{P}_i(x_i,\lambda_i))_\fU$ is a continuous homogeneous polynomial of degree $k$  and $\Vert \tilde{P} \Vert =\displaystyle\lim_{i,\fU} \Vert \tilde{P}_i \Vert =\displaystyle\lim_{i,\fU} \Vert P_i \Vert$.

Via the identification $(X_i \oplus_\infty \zC)_\fU=(X_i)_\fU \oplus_\infty \zC$ given by $(x_i,\lambda_i)_\fU=((x_i)_\fU,\displaystyle\lim_{i,\fU} \lambda_i)$ we have that the polynomial $Q:(X_i)_\fU \oplus_\infty \zC\rightarrow \zC$ defined as $Q((x_i)_\fU,\lambda)=\tilde{P}((x_i,\lambda)_\fU)$ is a continuous homogeneous polynomial of degree $k$ and $\Vert Q\Vert =\Vert \tilde{P}\Vert$. Then, the polynomial $P((x_i)_\fU)=Q((x_i)_\fU,1)$ is a continuous polynomial of degree at most $k$ and $\Vert P\Vert =\Vert Q\Vert =\displaystyle\lim_{i,\fU} \Vert P_i \Vert$. If $\displaystyle\lim_{i,\fU} \Vert P_{i,k} \Vert =0 $ then the degree of $P$ is lower than $k$.

\end{proof}

Note that, in the last proof, we can take the same approach used for non homogeneous polynomials in the real case, but we would not have the same control over the norms.

\section{ Main result }

This section contains our main result. As mentioned above, this result is partially motivated by Theorem 3.2 from \cite{LR}. We follow similar ideas for the proof. First, let us fix some notation that will be used throughout this section.

In this section, all polynomials considered are continuous scalar polynomials. Given a Banach space $X$, $B_X$ and $S_X$ denote the unit ball and the unit sphere of $X$ respectively, and $X^*$ is the dual of $X$. Given a polynomial $P$ on $X$, $deg(P)$ stands for the degree of $P$. 

\begin{defin} For a Banach space $X$ let $D(X,k_1,\cdots,k_n)$ denote the smallest constant that satisfies (\ref{problema}) for polynomials of degree $k_1,\cdots,k_n$. We also define $C(X,k_1,\cdots,k_n)$ as the smallest constant that satisfies (\ref{problema}) for homogeneous polynomials of degree $k_1,\cdots,k_n$.
\end{defin}

Throughout this section most of the results will have two parts. The first involving the constant $C(X,k_1,\cdots,k_n)$ for homogeneous polynomials and the second involving the constant $D(X,k_1,\cdots,k_n)$ for arbitrary polynomials. Given that the proof of both parts are almost equal, we will limit to prove only the second part of the results.

Recall that a space $X$ has the $1 +$ uniform approximation property if for all $n\in \zN$, exists $m=m(n)$ such that for every subspace $M\subset X$ with $dim(M)=n$ and every $\varepsilon > 0$ there is an operator $T\in \mathcal{L}(X,X)$ with $T|_M=id$, $rg(T)\leq m$ and $\Vert T\Vert  \leq 1 + \varepsilon$ (i.e. for every $\varepsilon > 0$ $X$ has the $1+\varepsilon$ uniform approximation property).

\begin{mainthm}\label{main thm} If $\fU$ is an ultrafilter on a family $I$ and $(X_i)_\fU$ is an ultraproduct of complex Banach spaces then

\begin{enumerate}
\item[(a)] $C((X_i)_\fU,k_1,\cdots,k_n) \geq \displaystyle\lim_{i,\fU}(C(X_i,k_1,\cdots,k_n)).$

\item[(b)] $D((X_i)_\fU,k_1,\cdots,k_n) \geq \displaystyle\lim_{i,\fU}(D(X_i,k_1,\cdots,k_n)).$
\end{enumerate}
Moreover, if each $X_i$ has the $1+$ uniform approximation property, equality holds in both cases.
\end{mainthm}

In order to prove this Theorem some auxiliary lemmas are going to be needed. The first one is due to Heinrich \cite{H}.

\begin{lem}\label{aprox} Given an ultraproduct of Banach spaces $(X_i)_\fU$, if each $X_i$ has the $1+$ uniform approximation property then $(X_i)_\fU$ has the metric approximation property.
\end{lem}

When working with the constants $C(X,k_1,\cdots,k_n)$ and $D(X,k_1,\cdots,k_n)$, the following characterization may result handy.

\begin{lem}\label{alternat} a) The constant $C(X,k_1,\cdots,k_n)$ is the biggest constant $M$ such that given any $\varepsilon >0$ there exist a set of homogeneous continuous polynomials $\{P_j\}_{j=1}^n$ with $deg(P_j)\leq k_j$ such that

\begin{equation} \label{condition} M\left \Vert \prod_{j=1}^{n} P_j \right \Vert \leq  (1+\varepsilon) \prod_{j=1}^{n} \Vert P_j \Vert. \end{equation}

\begin{flushleft}
b) The constant $D(X,k_1,\cdots,k_n)$ is the biggest constant satisfying the same for arbitrary polynomials.
\end{flushleft}

\end{lem}

To prove this Lemma it is enough to see that $D(X,k_1,\cdots,k_n)$ is decreasing as a function of the degrees $k_1,\cdots, k_n$ and use that the infimum is the greatest lower bound.

\begin{rem}\label{rmkalternat} It is clear that in Lemma \ref{alternat} we can take the polynomials $\{P_j\}_{j=1}^n$ with $deg(P_j)= k_j$ instead of $deg(P_j)\leq k_j$. Later on we will use both versions of the Lemma. 
\end{rem}

One last lemma is needed for the proof of the Main Theorem.

\begin{lem}\label{normas} Let $P$ be a (not necessarily homogeneous) polynomial on a complex Banach space $X$ with $deg(P)=k$. For any point $x\in X$ 
\begin{equation} |P(x)|\leq  \max\{\Vert x \Vert, 1\}^k \Vert P\Vert  . \nonumber \end{equation}
\end{lem}

\begin{proof} If $P$ is homogeneous the result is rather obvious since we have the inequality
\begin{equation} |P(x)|\leq \Vert x \Vert^k \Vert P\Vert . \nonumber \end{equation}
Suppose that $P=\sum_{l=0}^k P_l$ with $P_l$ an  $l-$homogeneous polynomial. Consider the space $X \oplus_\infty \zC$ and the polynomial $\tilde{P}:X \oplus_\infty \zC\rightarrow \zC$ defined by $\tilde{P}(x,\lambda)=\sum_{l=0}^k P_l(x)\lambda^{k-l}$. The polynomial $\tilde{P}$ is homogeneous of degree $k$ and $\Vert P \Vert = \Vert \tilde{P} \Vert $. Then, using that $\tilde{P}$ is homogeneous we have
\begin{equation} |P(x)|=|\tilde{P} (x,1)| \leq \Vert (x,1) \Vert^k \Vert \tilde{P} \Vert = \max\{\Vert x \Vert, 1\}^k \Vert P\Vert . \nonumber \end{equation}
\end{proof}

We are now able  to prove our main result.

\begin{proof}[Proof of Main Theorem] 
Throughout this proof we regard the space $(\zC)_\fU$ as $\zC$ via the identification $(\lambda_i)_\fU=\displaystyle\lim_{i,\fU} \lambda_i$.

First, we are going to see that $D((X_i)_\fU,k_1,\cdots,k_n) \geq \displaystyle\lim_{i,\fU}(D(X_i,k_1,\cdots,k_n))$. To do this we only need to prove that $\displaystyle\lim_{i,\fU}(D(X_i,k_1,\cdots,k_n))$ satisfies (\ref{condition}). Given $\varepsilon >0$ we need to find a set of polynomials $\{P_{j}\}_{j=1}^n$ on $(X_i)_\fU$ with $deg(P_{j})\leq k_j$  such that
$$ \displaystyle\lim_{i,\fU}(D(X_i,k_1,\cdots,k_n)) \left \Vert \prod_{j=1}^{n} P_j \right \Vert \leq (1+\varepsilon) \prod_{j=1}^{n} \left \Vert P_j \right \Vert .$$

By Remark \ref{rmkalternat} we know that for each $i\in I$ there is a set of polynomials $\{P_{i,j}\}_{j=1}^n$ on $X_i$ with $deg(P_{i,j})=k_j$ such that
$$ D(X_i,k_1,\cdots,k_n) \left \Vert \prod_{j=1}^{n} P_{i,j} \right \Vert \leq (1 +\varepsilon)\prod_{j=1}^{n} \left \Vert P_{i,j} \right \Vert.$$
Replacing $P_{i,j}$ with $P_{i,j}/\Vert P_{i,j} \Vert$ we may assume that $\Vert P_{i,j} \Vert =1$. Define the polynomials $\{P_j\}_{j=1}^n$ on $(X_i)_\fU$ by $P_j((x_i)_\fU)=(P_{i,j}(x_i))_\fU$. Then, by Proposition \ref{pollim}, $deg(P_j)\leq k_j$ and
\begin{eqnarray} \displaystyle\lim_{i,\fU}(D(X_i,k_1,\cdots,k_n)) \left \Vert \prod_{j=1}^{n} P_{j} \right \Vert &=& \displaystyle\lim_{i,\fU} \left(D(X_i,k_1,\cdots,k_n)\left \Vert \prod_{j=1}^{n} P_{i,j} \right \Vert \right)  \nonumber \\
&\leq& \displaystyle\lim_{i,\fU}\left((1+\varepsilon)\prod_{j=1}^{n}\Vert  P_{i,j}  \Vert \right)\nonumber \\
&=& (1+\varepsilon)\prod_{j=1}^{n} \Vert P_{j}  \Vert \nonumber 
 \nonumber 
\end{eqnarray}
as desired.

To prove that $D((X_i)_\fU,k_1,\cdots,k_n) \leq \displaystyle\lim_{i,\fU}(D(X_i,k_1,\cdots,k_n))$ if each $X_i$ has the $1+$ uniform approximation property is not as straightforward. Given $\varepsilon >0$, let $\{P_j\}_{j=1}^n$ be a set of polynomials on $(X_i)_\fU$ with $deg(P_j)=k_j$ such that
$$ D((X_i)_\fU,k_1,\cdots,k_n) \left \Vert \prod_{j=1}^{n} P_j \right \Vert \leq  (1+\varepsilon)\prod_{j=1}^{n} \Vert P_j \Vert . $$

Let $K\subseteq B_{(X_i)_\fU}$ be the finite set $K=\{x_1,\cdots, x_n\}$ where $ x_j$ is such that 
$$|P_j(x_j)| > \Vert P_j\Vert (1- \varepsilon) \mbox{ for }j=1,\cdots, n.$$ 
Being that each $X_i$ has the $1+$ uniform approximation property, then, by Lemma \ref{aprox}, $(X_i)_\fU$ has the metric approximation property. Therefore, exist a finite rank operator $S:(X_i)_\fU\rightarrow (X_i)_\fU$ such that $\Vert S\Vert \leq 1 $ and 
$$\Vert P_j - P_j \circ S \Vert_K< |P_j(x_j)|\varepsilon \mbox{ for }j=1,\cdots, n.$$

Now, define the polynomials $Q_1,\cdots, Q_n$ on $(X_i)_\fU$ as $Q_j=P_j\circ S$. Then 
$$\left\Vert \prod_{j=1}^n Q_j \right\Vert \leq \left\Vert \prod_{j=1}^n P_j \right\Vert $$
$$\Vert Q_j\Vert_K > | P_j(x_j)|-\varepsilon | P_j(x_j)| =| P_j(x_j)| (1-\varepsilon) \geq \Vert P_j \Vert(1-\varepsilon)^2.$$
The construction of this polynomials is a slight variation of Lemma 3.1 from \cite{LR}. We have the next inequality for the product of the polynomials $\{Q_j\}_{j=1}^n$
\begin{eqnarray} D((X_i)_\fU,k_1,\cdots,k_n)\left \Vert \prod_{j=1}^{n} Q_{j} \right \Vert &\leq& D((X_i)_\fU,k_1,\cdots,k_n)\left \Vert \prod_{j=1}^{n} P_{j} \right \Vert \nonumber \\
&\leq&  (1+\varepsilon) \prod_{j=1}^{n}  \left \Vert P_{j} \right \Vert . \label{desq}\end{eqnarray}

Since $S$ is a finite rank operator, the polynomials $\{ Q_j\}_{j=1}^n$ have the advantage that are finite type polynomials. This will allow us to construct polynomials on $(X_i)_\fU$ which are limit of polynomials on the spaces $X_i$. For each $j$ write $Q_j=\sum_{t=1}^{m_j}(\psi_{j,t})^{r_{j,t}}$ with $\psi_{j,t}\in (X_i)_\fU^*$, and consider the spaces $N=\rm{span}  \{x_1,\cdots,x_n\}\subset (X_i)_\fU$ and $M=\rm{span} \{\psi_{j,t} \}\subset (X_i)_\fU^*$. By the local duality of ultraproducts (see Theorem 7.3 from \cite{H}) exist $T:M\rightarrow (X_i^*)_\fU$ an $(1+\varepsilon)-$isomorphism such that
$$JT(\psi)(x)=\psi(x) \mbox{ } \forall x\in N, \mbox{ } \forall \psi\in M$$
where $J:(X_i^*)_\fU\rightarrow (X_i)_\fU^*$ is the canonical embedding. Let $\phi_{j,t}=JT(\psi_{j,t})$ and consider the polynomials $\bar{Q}_1,\cdots, \bar{Q}_n$ on $(X_i)_\fU$ with $\bar{Q}_j=\sum_{t=1}^{m_j}(\phi_{j,t})^{r_{j,t}}$. Clearly $\bar{Q}_j$ is equal to $Q_j$ in $N$ and $K\subseteq N$, therefore we have the following lower bound for the norm of each polynomial
\begin{equation}\Vert \bar{Q}_j \Vert \geq \Vert \bar{Q}_j \Vert_K = \Vert Q_j \Vert_K >\Vert P_j \Vert(1-\varepsilon)^2 \label{desbarq} \end{equation}

Now, let us find an upper bound for the norm of the product $\Vert \prod_{j=1}^n \bar{Q}_j \Vert$. Let $x=(x_i)_\fU$ be any point in $B_{(X_i)_\fU}$. Then, we have 
\begin{eqnarray} \left|\prod_{j=1}^n \bar{Q}_j(x)\right| &=& \left|\prod_{j=1}^n \sum_{t=1}^{m_j}(\phi_{j,t} (x))^{r_{j,t}}\right|=\left|\prod_{j=1}^n \sum_{t=1}^{m_j} (JT\psi_{j,t}(x))^{r_{j,t}} \right| \nonumber \\
&=& \left|\prod_{j=1}^n \sum_{t=1}^{m_j}((JT)^*\hat{x}(\psi_{j,t}))^{r_{j,t}}\right|\nonumber\end{eqnarray}

Since $(JT)^*\hat{x}\in M^*$, $\Vert (JT)^*\hat{x}\Vert =\Vert JT \Vert \Vert x \Vert \leq \Vert J \Vert \Vert T \Vert \Vert x \Vert< 1 + \varepsilon$ and $M^*=\frac{(X_i)_\fU^{**}}{M^{\bot}}$, we can chose $z^{**}\in (X_i)_\fU^{**}$ with $\Vert z^{**} \Vert < \Vert (JT)^*\hat{x}\Vert+\varepsilon < 1+2\varepsilon$, such that $\prod_{j=1}^n \sum_{t=1}^{m_j} ((JT)^*\hat{x}(\psi_{j,t}))^{r_{j,t}}= \prod_{j=1}^n \sum_{t=1}^{m_j} (z^{**}(\psi_{j,t}))^{r_{j,t}}$. By Goldstine's Theorem exist a net $\{z_\alpha\} \subseteq (X_i)_\fU$ $w^*-$convergent to $z$ in $(X_i)_\fU^{**}$ with $\Vert z_\alpha \Vert = \Vert z^{**}\Vert$. In particular, $ \psi_{j,t}(z_\alpha)$ converges to $z^{**}(\psi_{j,t})$. If we call $\kkk = \sum k_j$, since $\Vert z_\alpha \Vert< (1+2\varepsilon)$, by Lemma \ref{normas}, we have
\begin{equation} \left \Vert \prod_{j=1}^{n} Q_j \right \Vert (1+2\varepsilon)^\kkk \geq \left|\prod_{j=1}^n Q_j(z_\alpha)\right| = \left|\prod_{j=1}^n \sum_{t=1}^{m_j} ((\psi_{j,t})(z_\alpha))^{r_{j,t}}\right| .      \label{usecomplex} 
\end{equation}
Combining this with the fact that
\begin{eqnarray} \left|\prod_{j=1}^{n} \sum_{t=1}^{m_j} ((\psi_{j,t})(z_\alpha))^{r_{j,t}}\right| &\longrightarrow&  \left|\prod_{j=1}^{n} \sum_{t=1}^{m_j} (z^{**}(\psi_{j,t}))^{r_{j,t}}\right|\nonumber\\
 &=& \left|\prod_{j=1}^{n} \sum_{t=1}^{m_j} ((JT)^*\hat{x}(\psi_{j,t}))^{r_{j,t}}\right| = \left|\prod_{j=1}^{n} \bar{Q}_j(x)\right|\nonumber
\end{eqnarray}
we conclude that $\left \Vert \prod_{j=1}^{n} Q_j \right \Vert (1+2\varepsilon)^\kkk \geq |\prod_{j=1}^{n} \bar{Q}_j(x)|$.

Since the choice of $x$ was arbitrary we arrive to the next inequality
\begin{eqnarray}
 D((X_i)_\fU,k_1,\cdots,k_n)\left \Vert \prod_{j=1}^{n} \bar{Q}_j \right \Vert &\leq& (1+2\varepsilon)^\kkk D((X_i)_\fU,k_1,\cdots,k_n)   \left \Vert \prod_{j=1}^{n} Q_j \right \Vert   \nonumber \\
&\leq&  (1+2\varepsilon)^\kkk (1+\varepsilon) \prod_{j=1}^{n}  \left \Vert P_{j} \right \Vert \label{desbarq2} \\
&<& (1+2\varepsilon)^\kkk (1+\varepsilon) \frac{\prod_{j=1}^{n} \Vert \bar{Q}_j \Vert }{(1-\varepsilon)^{2n}} .\label{desbarq3}  \
\end{eqnarray}
In (\ref{desbarq2}) and (\ref{desbarq3}) we use (\ref{desq}) and (\ref{desbarq}) respectively. The polynomials $\bar{Q}_j$ are not only of finite type, these polynomials are also generated by elements of $(X_i^*)_\fU$. This will allow us to write them as limits of polynomials in $X_i$. For any $i$, consider the polynomials $\bar{Q}_{i,1},\cdots,\bar{Q}_{i,n}$ on $X_i$ defined by $\bar{Q}_{i,j}= \displaystyle\sum_{t=1}^{m_j} (\phi_{i,j,t})^{r_{j,t}}$, where the functionals $\phi_{i,j,t}\in X_i^*$ are such that $(\phi_{i,j,t})_\fU=\phi_{j,t}$. Then $\bar{Q}_j(x)=\displaystyle\lim_{i,\fU} \bar{Q}_{i,j}(x)$ $\forall x \in (X_i)_\fU$ and, by Proposition \ref{pollim}, $\Vert \bar{Q}_j \Vert = \displaystyle\lim_{i,\fU} \Vert \bar{Q}_{i,j} \Vert$. Therefore 
\begin{eqnarray}
D((X_i)_\fU,k_1,\cdots,k_n) \displaystyle\lim_{i,\fU} \left \Vert \prod_{j=1}^{n} \bar{Q}_{i,j} \right \Vert &=& D((X_i)_\fU,k_1,\cdots,k_n) \left \Vert \prod_{j=1}^{n} \bar{Q}_{j} \right \Vert \nonumber \\
&<& \frac{(1+\varepsilon)(1+2\varepsilon)^\kkk}{(1-\varepsilon)^{2n}} \prod_{j=1}^{n} \Vert \bar{Q}_{j}  \Vert \nonumber \\
&=& \frac{(1+\varepsilon)(1+2\varepsilon)^\kkk	}{(1-\varepsilon)^{2n}} \prod_{j=1}^{n} \displaystyle\lim_{i,\fU} \Vert \bar{Q}_{i,j}  \Vert .
 \nonumber 
\end{eqnarray}

To simplify the notation let us call $\lambda = \frac{(1+\varepsilon)(1+2\varepsilon)^\kkk}{(1-\varepsilon)^{2n}} $. Take $L>0$ such that 

\begin{equation}D((X_i)_\fU,k_1,\cdots,k_n) \displaystyle\lim_{i,\fU} \left \Vert \prod_{j=1}^{n} \bar{Q}_{i,j} \right \Vert < L < \lambda \prod_{j=1}^{n} \displaystyle\lim_{i,\fU} \Vert \bar{Q}_{i,j}  \Vert . \nonumber \end{equation}

Since $(-\infty, \frac{L}{D((X_i)_\fU,k_1,\cdots,k_n)})$ and $(\frac{L}{\lambda},+\infty)$ are neighborhoods of $\displaystyle\lim_{i,\fU} \left \Vert \prod_{j=1}^{n} \bar{Q}_{i,j} \right \Vert$ and $\prod_{j=1}^{n} \displaystyle\lim_{i,\fU} \Vert \bar{Q}_{i,j}  \Vert$ respectively, and $\prod_{j=1}^{n} \displaystyle\lim_{i,\fU} \Vert \bar{Q}_{i,j}  \Vert= \displaystyle\lim_{i,\fU} \prod_{j=1}^{n} \Vert \bar{Q}_{i,j}  \Vert$, by definition of $\displaystyle\lim_{i,\fU}$, the sets
$$A=\{i_0: D((X_i)_\fU,k_1,\cdots,k_n) \left \Vert \prod_{j=1}^{n} \bar{Q}_{i_0,j} \right \Vert <L\} \mbox{ and }B=\{i_0: \lambda \prod_{j=1}^{n}  \Vert \bar{Q}_{i_0,j}  \Vert > L \}$$
are elements of $\fU$. Since $\fU$ is closed by finite intersections $A\cap B\in \fU$. If we take any element $i_0 \in A\cap B$ then, for any $\delta >0$, we have that
\begin{equation}D((X_i)_\fU,k_1,\cdots,k_n) \left \Vert \prod_{j=1}^{n} \bar{Q}_{i_0,j} \right \Vert \frac{1}{\lambda}\leq \frac{L}{\lambda} \leq \prod_{j=1}^{n}  \Vert \bar{Q}_{i_0,j}  \Vert < (1+ \delta)\prod_{j=	1}^{n}  \Vert \bar{Q}_{i_0,j} \Vert \nonumber \end{equation}
Then, since $\delta$ is arbitrary, the constant $D((X_i)_\fU,k_1,\cdots,k_n)\frac{1}{\lambda}$ satisfy (\ref{condition}) for the space $X_{i_0}$ and therefore, by Lemma \ref{alternat},
\begin{equation} \frac{1}{\lambda}D((X_i)_\fU,k_1,\cdots,k_n) \leq  D(X_{i_0},k_1,\cdots,k_n). 	\nonumber \end{equation}

This holds true for any $i_0$ in $A\cap B$. Since $A\cap B \in \fU$, by Lemma \ref{lemlimit}, $\frac{1}{\lambda}D((X_i)_\fU,k_1,\cdots,k_n)\leq \displaystyle\lim_{i,\fU}  D(X_i,k_1,\cdots,k_n) $. Using that $\lambda \rightarrow 1$ when $\varepsilon \rightarrow 0$ we conclude that $D((X_i)_\fU,k_1,\cdots,k_n)\leq \displaystyle\lim_{i,\fU}  D(X_i,k_1,\cdots,k_n).$
\end{proof}

Similar to Corollary 3.3 from \cite{LR}, a straightforward corollary of our main result is that for any complex Banach space $X$ with $1+$ uniform approximation property $C(X,k_1,\cdots,k_n)=C(X^{**},k_1,\cdots,k_n)$ and $D(X,k_1,\cdots,k_n)=D(X^{**},k_1,\cdots,k_n)$ .  Using that $X^{**}$ is $1-$complemented in some adequate ultrafilter $(X)_{\fU}$ the result is rather obvious. For a construction of the adequate ultrafilter see \cite{LR}. 

But following the previous proof, and using the principle of local reflexivity applied to $X^*$ instead of the local duality of ultraproducts, we can prove the next stronger result.

\begin{thm}  Let $X$ be a complex Banach space. Then

\begin{enumerate}
\item[(a)] $C(X^{**},k_1,\cdots,k_n)\geq C(X,k_1,\cdots,k_n).$

\item[(b)] $D(X^{**},k_1,\cdots,k_n \geq D(X,k_1,\cdots,k_n)).$
\end{enumerate}
Moreover, if $X^{**}$ has the metric approximation property, equality holds in both cases.
\end{thm}

\begin{proof} The inequality $D(X^{**},k_1,\cdots,k_n) \geq D(X,k_1,\cdots,k_n)$ is a corollary of Theorem \ref{main thm} (using the adequate ultrafilter mentioned above).

Let us prove that if $X^{**}$ has the metric approximation property then $D((X^{**},k_1,\cdots,k_n)\geq D(X,k_1,\cdots,k_n)$. Given $\varepsilon >0$, let $\{P_j\}_{j=1}^n$ be a set of polynomials on $X^{**}$ with $deg(P_j)=k_j$ such that
\begin{equation} D(X^{**},k_1,\cdots,k_n)\left \Vert \prod_{j=1}^{n} P_{j} \right \Vert \leq (1+\varepsilon)\prod_{j=1}^{n} \left \Vert P_{j} \right \Vert .\nonumber \end{equation}

Analogous to the proof of Theorem \ref{main thm}, since $X^{**}$ has the metric approximation, we can construct finite type polynomials $Q_1,\cdots,Q_n$ on $X^{**}$ with $deg(Q_j)=k_j$, $\Vert Q_j \Vert_K \geq \Vert P_j \Vert (1-\varepsilon)^2$ for some finite set $K\subseteq B_{X^{**}}$ and that
\begin{equation}D(X^{**},k_1,\cdots,k_n)\left \Vert \prod_{j=1}^{n} Q_{j} \right \Vert < (1+\varepsilon)\prod_{j=1}^{n} \left \Vert P_{j} \right  \Vert . \nonumber \end{equation}

Suppose that $Q_j=\sum_{t=1}^{m_j}(\psi_{j,t})^{r_{j,t}}$ and consider the spaces $N=\rm{span} \{K\}$ and $M=\rm{span} \{\psi_{j,t} \}$. By the principle of local reflexivity (see \cite{D}), applied to $X^*$ (thinking $N$ as a subspaces of $(X^*)^*$ and $M$ as a subspaces of $(X^*)^{**}$), there is an  $(1+\varepsilon)-$isomorphism $T:M\rightarrow X^*$ such that
$$JT(\psi)(x)=\psi(x) \mbox{ } \forall x\in N, \mbox{ } \forall \psi\in M\cap X^*=M,$$
where $J:X^*\rightarrow X^{***}$ is the canonical embedding. 

Let $\phi_{j,t}=JT(\psi_{j,t})$ and consider the polynomials $\bar{Q}_1,\cdots, \bar{Q}_n$ on $X^{**}$ defined by $\bar{Q}_j=\sum_{t=1}^{m_j}(\phi_{j,t})^{r_{j,t}}$. Following the proof of the Main Theorem, one arrives to the inequation
\begin{equation}D(X^{**},k_1,\cdots,k_n)\left \Vert \prod_{j=1}^{n} \bar{Q_j} \right \Vert < (1+ \delta) \frac{(1+\varepsilon)(1+2\varepsilon)^\kkk}{(1-\varepsilon)^{2n}} \prod_{j=1}^{n} \Vert \bar{Q_j}  \Vert \nonumber \end{equation}
for every $\delta >0$. Since each $\bar{Q}_j$ is generated by elements of $J(X^*)$, by Goldstine's Theorem, the restriction of $\bar{Q}_j$ to $X$ has the same norm and the same is true for $\prod_{j=1}^{n} \bar{Q_j}$. Then
\begin{equation}D(X^{**},k_1,\cdots,k_n)\left \Vert \prod_{j=1}^{n} \left.\bar{Q_j}\right|_X \right \Vert < (1+ \delta) \frac{(1+\varepsilon)(1+2\varepsilon)^\kkk}{(1-\varepsilon)^{2n}} \prod_{j=1}^{n} \Vert \left.\bar{Q_j}\right|_X  \Vert \nonumber \end{equation}
By Lemma \ref{alternat} we conclude that 
$$\frac{(1-\varepsilon)^{2n}}{(1+\varepsilon)(1+2\varepsilon)^\kkk}D(X^{**},k_1,\cdots,k_n)\leq D(X,k_1,\cdots,k_n).$$ Given that the choice of $\varepsilon$ is arbitrary and that $\frac{(1-\varepsilon)^{2n}}{(1+\varepsilon)(1+2\varepsilon)^\kkk} $ tends to $1$ when $\varepsilon$ tends to $0$ we conclude that $D(X^{**},k_1,\cdots,k_n)\leq D(X,k_1,\cdots,k_n)$.

\end{proof}

Note that in the proof of the Main Theorem the only parts where we need the spaces to be complex Banach spaces are at the beginning, where we use Proposition \ref{pollim}, and in the inequality (\ref{usecomplex}), where we use Lemma \ref{normas}. But both results holds true for homogeneous polynomials on a real Banach space. Then, copying the proof of the Main Theorem we obtain the following result for real spaces.

\begin{thm} If $\fU$ is an ultrafilter on a family $I$ and $(X_i)_\fU$ is an ultraproduct of real Banach spaces then
 $$C((X_i)_\fU,k_1,\cdots,k_n) \geq \displaystyle\lim_{i,\fU}(C(X_i,k_1,\cdots,k_n)).$$

If in addition each $X_i$ has the $1+$ uniform approximation property, the equality holds.
\end{thm}

Also we can get a similar result for the bidual of a real space.

\begin{thm}  Let $X$ be a real Banach space. Then

\begin{enumerate}
\item[(a)] $C(X^{**},k_1,\cdots,k_n)\geq C(X,k_1,\cdots,k_n).$

\item[(b)] $D(X^{**},k_1,\cdots,k_n) \geq D(X,k_1,\cdots,k_n).$
\end{enumerate}
If $X^{**}$ has the metric approximation property, equality holds in $(a)$.
\end{thm}

\begin{proof} The proof of item $(a)$ is the same that in the complex case, so we limit to prove $D(X^{**},k_1,\cdots,k_n) \geq D(X,k_1,\cdots,k_n))$. To do this we will show that given an arbitrary $\varepsilon >0$, there is a set of polynomials $\{P_{j}\}_{j=1}^n$ on $X^{**}$ with $deg(P_{j})\leq k_j$  such that
$$ D(X,k_1,\cdots,k_n) \left \Vert \prod_{j=1}^{n} P_j \right \Vert \leq (1+\varepsilon) \prod_{j=1}^{n} \left \Vert P_j \right \Vert .$$

Take $\{Q_{j}\}_{j=1}^n$ a set of polynomials  on $X$ with $deg(Q_j)=k_j$ such that
$$ D(X,k_1,\cdots,k_n) \left \Vert \prod_{j=1}^{n} Q_{j} \right \Vert \leq (1 +\varepsilon)\prod_{j=1}^{n} \left \Vert Q_{j} \right \Vert.$$

Consider now the polynomials $P_j=AB(Q_j)$, where $AB(Q_j)$ is the Aron Berner extension of $Q_j$ (for details on this extension see \cite{AB} or \cite{Z}). Since $AB\left( \prod_{j=1}^n P_j \right)=\prod_{j=1}^n AB(P_j)$, using that the Aror Berner extension preserves norm (see \cite{DG}) we have

\begin{eqnarray} D(X,k_1,\cdots,k_n) \left \Vert \prod_{j=1}^{n} P_{j} \right \Vert &=& D(X,k_1,\cdots,k_n) \left \Vert \prod_{j=1}^{n} Q_{j}  \right \Vert\nonumber \\
&\leq& (1 +\varepsilon)\prod_{j=1}^{n} \left\Vert Q_{j} \right\Vert \nonumber \\
&=& (1 +\varepsilon)\prod_{j=1}^{n} \left \Vert P_{j} \right  \Vert \nonumber 
\end{eqnarray}
as desired.

\end{proof}

As a final remark, we mention two types of spaces for which the results on this section can be applied.

Corollary 9.2 from \cite{H} states that any Orlicz space $L_\Phi(\mu)$, with $\mu$ a finite measure and $\Phi$ an Orlicz function with regular variation at $\infty$, has the $1+$ uniform projection property, which is stronger than  the $1+$ uniform approximation property.

In \cite{PeR} Section two, A. Pe\l czy\'nski and H. Rosenthal proved that any $\cL_{p,\lambda}-$space ($1\leq \lambda < \infty$) has the $1+\varepsilon-$uniform projection property for every $\varepsilon>0$ (which is stronger than the $1+\varepsilon-$uniform approximation property), therefore, any $\cL_{p,\lambda}-$space has the $1+$ uniform approximation property. 

\section*{Acknowledgment}
I would like to thank Professor Daniel Carando for both encouraging me to write this article, and for his comments and remarks which improved its presentation and content.


\begin{thebibliography}{HD}




\normalsize
\baselineskip=17pt


\bibitem[A]{A} R. M. J. Arias-de-Reyna.
\emph{Gaussian variables, polynomials and permanents}.
Linear Algebra Appl. 285 (1998), 107--114.

\bibitem[AB]{AB} R. M. Aron and P. D. Berner.
\emph{A Hahn-Banach extension theorem for analytic mapping}.
Bull. Soc. Math. France 106 (1978), 3--24.

\bibitem[BST]{BST} C. Ben\'{\i}tez, Y. Sarantopoulos and A. Tonge.
\emph{Lower bounds for norms of products of polynomials}.
Math. Proc. Cambridge Philos. Soc. 124 (1998), 395--408.



\bibitem[CPR]{CPR} D. Carando, D. Pinasco y J.T. Rodr\'{\i}guez.
\emph{Lower bounds for norms of products of polynomials on $L_p$ spaces}.
Studia Math. 214 (2013), 157--166.

\bibitem[DG]{DG} A. M. Davie and T. W. Gamelin.
\emph{A theorem on polynomial-star approximation}.
Proc. Amer. Math. Soc. 106 (1989) 351--356.

\bibitem[D]{D} D. W. Dean.
\emph{The equation $L(E,X^{**})=L(E,X)^{**}$ and the principle of local reflexivity}.
Proceedings of the American Mathematical Society. 40 (1973), 146-148.

\bibitem[H]{H} S. Heinrich.
\emph{Ultraproducts in Banach space theory}.
J. Reine Angew. Math. 313 (1980), 72--104.

\bibitem[LR]{LR} M. Lindstr\"{o}m and R. A. Ryan.
\emph{Applications of ultraproducts to infinite dimensional holomorphy}.
Math. Scand. 71 (1992), 229--242.

\bibitem[PeR]{PeR} A. Pe\l czy\'nski and H. Rosenthal.
\emph{Localization techniques in $L_p$ spaces}.
Studia Math. 52 (1975), 265--289.

\bibitem[P]{P} D. Pinasco.
\emph{Lower bounds for norms of products of polynomials via Bombieri inequality}.
Trans. Amer. Math. Soc. 364 (2012), 3993--4010.

\bibitem[Z]{Z} I. Zalduendo.
\emph{Extending polynomials on Banach Spaces - A survey}.
 Rev. Un. Mat. Argentina 46 (2005), 45--72.

\end{thebibliography}
\end{document}